\documentclass[10pt]{article}

\usepackage{amsmath,amsthm,graphicx,amsfonts,lineno}

\newtheorem{theorem}{Theorem}[section]

\newtheorem{lemma}[theorem]{Lemma}

\title{The Erd\H os-S\'os Conjecture for Geometric Graphs}

\author{
  Luis F.~Barba \thanks{Instituto de Matem\'aticas, UNAM}
  \and
  Ruy Fabila-Monroy \thanks{Departamento de Matem\'aticas, CINVESTAV.} \footnote{Partially supported by CONACYT of
Mexico, Grant 153984.}
  \and
  Dolores Lara \thanks{Departament de Matem\`atica Aplicada II, Universitat Polit\`ecnica de Catalunya (UPC).} 
  \footnotemark[3]
  \and
  Jes\'us Lea\~nos\thanks{Escuela de Matem\'aticas, UAZ.}
  \and
  Cynthia Rodr\'iguez \thanks{Departamento de Sistemas, UAM-Azcapotzalco.}
  \and
  Gelasio Salazar \thanks{Instituto de F\'isica, UASLP. Partially supported by CONACYT of Mexico, Grant 106432.}
  \and
  Francisco Zaragoza \footnotemark[5]
}

\begin{document}

\maketitle

\begin{abstract}
Let $f(n,k)$ be the minimum number of edges
that must be removed from some complete geometric
graph $G$ on $n$ points, so that there
exists a tree on $k$ vertices that is no 
longer a planar subgraph of $G$.
In this paper we show that 
$\left( \frac{1}{2} \right )\frac{n^2}{k-1}-\frac{n}{2}\le f(n,k) \le 2 \frac{n(n-2)}{k-2}$.
For the case when $k=n$, we show that
$2 \le f(n,n) \le 3$. For the case when
$k=n$ and $G$ is a geometric graph on a set of points
in convex position, we show that at least
three edges must be removed.
\end{abstract}

\section{Introduction}

One of the most notorious problems in extremal graph theory is the
Erd\H{os}-S\'os Conjecture, which states that every simple graph with
average degree greater than $k - 2$ contains every tree on $k$
vertices as a subgraph. This conjecture was recently proved true for
all sufficiently large $k$ (unpublished work of Ajtai, Koml\'os,
Simonovits, and Szemer\'edi).

In this paper we investigate a variation of this conjecture in the
setting of geometric graphs. Recall that a {\em geometric graph}
$G$ consists of a set $S$ of points in the plane (these are the
vertices of $G$), plus a set of straight line segments, each of which joins
two points in $S$ (these are the edges of $G$).  In particular, any
set $S$ of points in the plane in general position naturally induces a
{complete geometric graph}. For brevity, we often refer to the edges
of this graph simply as edges {\em of} $S$. 
If $S$ is in convex position then $G$ is a
\emph{convex geometric graph}.
A geometric graph is {\em planar} if no two of its edges cross each
other. An \emph{embedding} of an abstract graph $H$ into a geometric
graph $G$ is an isomorphism from $H$ to a planar geometric
subgraph of $G$.
For $r\ge 0$, an \emph{$r$-edge} is an edge of $G$ such that
in one of the two open semi-planes defined
by the line containing it, there are exactly $r$ points of $G$.

In this paper all point sets are in general position and $G$ is a complete geometric graph on $n$ points.
It is well known that for every integer $1 \le k \le n$, $G$ contains every tree on $k$ vertices
as a planar subgraph \cite{optimal}.
Even more, it is possible to embed any such tree into $G$, when the image
of a given vertex is prespecified \cite{rooted}.

Let $F$ be a subset of edges of $G$, which we call \emph{forbidden edges}.
 If $T$ is a
tree for which every embedding into $G$
uses an edge of $F$, then we say that $F$ \emph{forbids} $T$.
 In this paper  we study the question of what is
the minimum size of $F$ so that there is a tree on $k$ vertices that 
is forbidden by $F$.  Let $f(n,k)$ be the minimum of this number
taken over all complete geometric graphs on
$n$ points. 
As $f(2,2)=1$, $f(3,3)=2$, $f(4,4)=2$ and $f(n,2)=\binom{n}{2}$, we assume through out
the paper that $n\ge 5$
and $k\ge 3$.

We show the following bounds
on $f(n,k)$.

\begin{theorem} \label{thm:main}
$$\left( \frac{1}{2} \right )\frac{n^2}{k-1}-\frac{n}{2}\le f(n,k) \le 2 \frac{n(n-2)}{k-2}$$
\end{theorem}

\begin{theorem}\label{thm:k=n}

$$2 \le f(n,n) \le 3$$
\end{theorem}

In the case when $G$ is a convex complete geometric graph, 
we show that the minimum number of edges needed to forbid 
a tree on $n$ vertices is three. 

An equivalent formulation of the problem studied
in this paper is to
ask how many edges must be removed from $G$ so that
it no longer contains \emph{some} planar subtree on $k$ vertices.
A related problem is to ask how many edges must be
removed from $G$ so that it no longer contains
\emph{any} planar subtree on $k$ vertices.
For the case of $k=n$, in \cite{ramseytype}, 
it is proved that if any $n-2$ edges are removed from $G$, 
it still contains a planar spanning subtree. Note that
if the $n-1$ edges incident to any vertex of $G$ 
are removed, then $G$ no longer contains a spanning subtree.
 In general,
for $2\le k \le n-1$, in \cite{edge-removal}, it
is proved that if any set of 
$\left \lceil \frac{n(n-k+1)}{2} \right \rceil-1$ edges
are removed from $G$, it still contains
a planar subtree on $k$ vertices. In the same
paper it is also shown that this bound is tight---a
 geometric graph on $n$ vertices and a subset
of $\left \lceil \frac{n(n-k+1)}{2} \right \rceil$
of its edges are shown, so that when these edges
are removed, every planar
subtree has at most $k-1$ vertices. In \cite{packing} the authors study the 
similingly unrelated problem of packing
two trees into planar graphs. That is, given two trees
on $n$ vertices, the authors consider the question
of when it is possible to find a planar graph
having both of them as spanning trees and in which
the trees are edge disjoint. However, although theirs is a combinatorial 
question rather than geometric, their Theorem 2.1 implies our Lemma~\ref{lem:few_0edges}. We provide a self
contained proof of Lemma~\ref{lem:few_0edges} for completeness.

A previous version of this paper appeared in the conference
proceedings of EUROCG'12~\cite{eurocg}.

\section{Spanning Trees}
In this section we consider the case when $k=n$.
Let $T$ be a tree on $n$ vertices. Consider
the following algorithm to embed $T$ into
$G$. Choose a vertex $v$ of $T$ and root
$T$ at $v$. For every vertex of $T$
choose an arbitrary order of its children.
Suppose that the neighbors of $v$
are $u_1,\dots,u_m$, and let 
$n_1,\dots,n_m$ be the number of nodes
in their corresponding subtrees.
Choose a convex hull point $p$ of $G$ and
embed $v$ into $p$. Sort
the remaining points of $G$ counter-clockwise by angle
around $p$. Choose $m+1$ rays centered
at $p$ so that the wedge between two consecutive
rays is convex and between the $i$-th ray
and the $(i+1)$-th ray there are exactly $n_i$ points of $G$.
Let $S_i$ be this set of points. For each
$u_i$ choose a convex hull vertex of $S_i$ visible
from $p$ and embed $u_i$ into this point. Recursively
embed the subtrees rooted at each $u_i$ into $S_i$.
Note that this algorithm provides an embedding of $T$
into $G$. We will use this embedding frequently throughout the paper.
See Figure~\ref{fig:embedding}.

For every integer $n\ge 2$ we define
a tree $T_n$ as follows:
If $n=2$, then $T_n$ consists of only one edge;
if $n$ is odd, then $T_n$ is constructed
by subdividing once every edge of a star
on $\frac{n-1}{2}$ vertices;
if $n$ is even and greater than $2$, then $T_n$ is constructed
by subdividing an edge of $T_{n-1}$.
These trees are particular cases of \emph{spider trees}.
See Figure~\ref{fig:matatenas}.

\begin{figure}
  \begin{center}
  \includegraphics[width=0.8\textwidth]{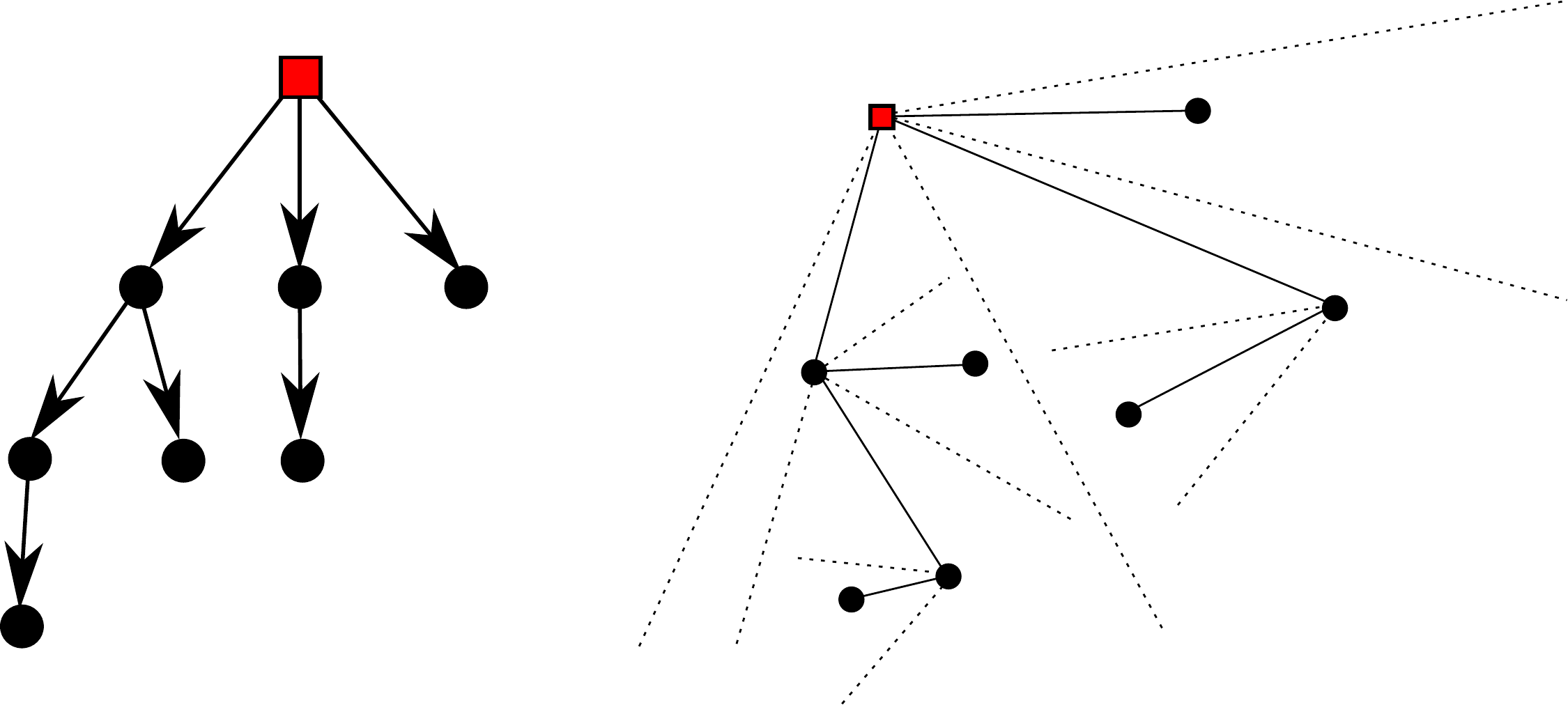}
  \caption{An embedding of a tree using the algorithm.}\label{fig:embedding}
\end{center}
\end{figure}

\begin{figure}
\begin{center}
  \includegraphics[width=0.8\textwidth]{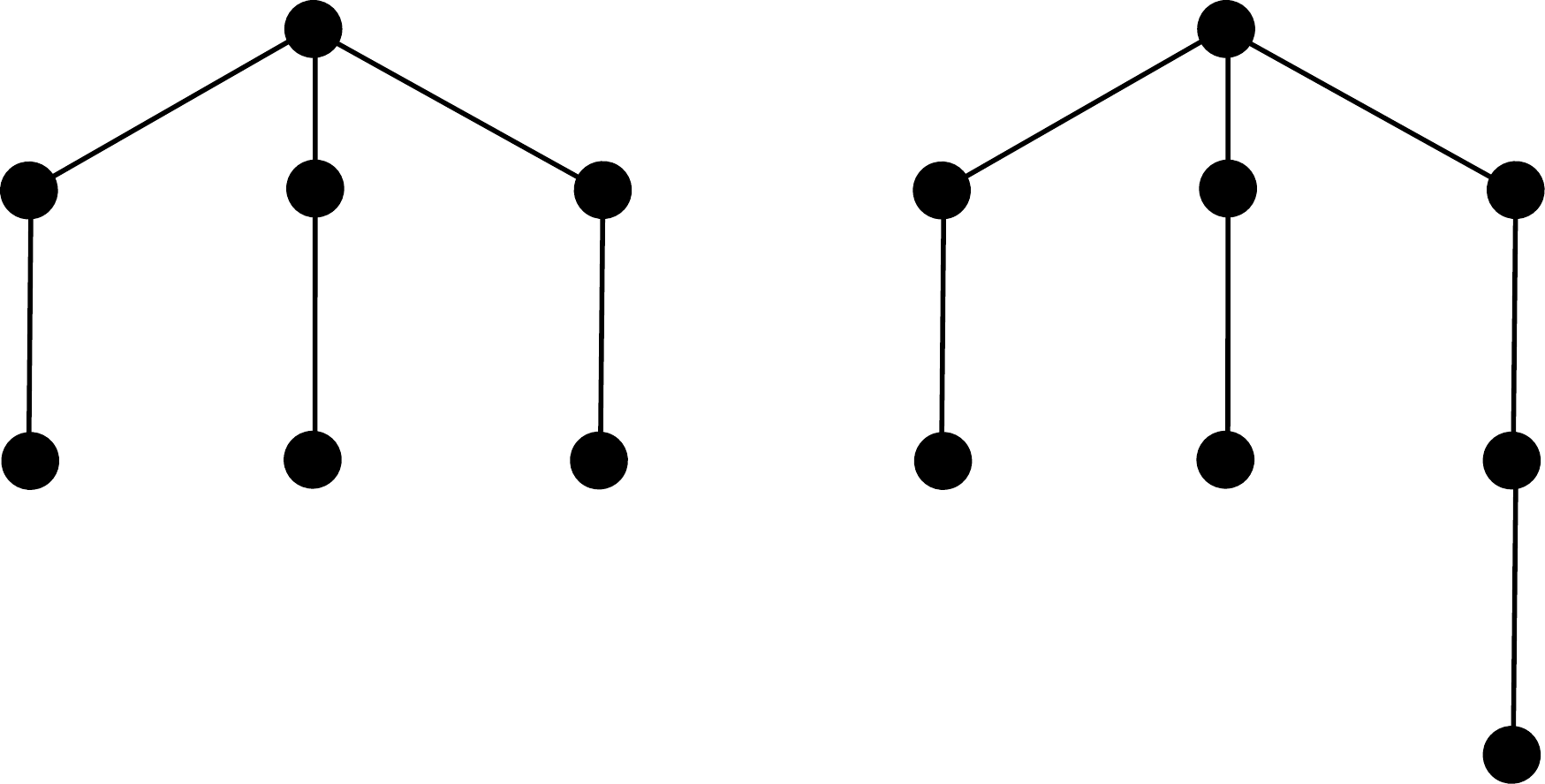}
  \caption{$T_7$ and $T_8$.}\label{fig:matatenas}
\end{center}
\end{figure}

We prove the lower bound of $f(n,n)\ge 2$ of Theorem~\ref{thm:k=n}.

\begin{theorem}
If $G$ has only one forbidden edge, then
any tree on $n$ vertices can be embedded into 
$G$, without using the forbidden edge.
\end{theorem}
\begin{proof}
Let $e$ be the forbidden edge of $G$.
Let $T$ be a tree on $n$ vertices.
Choose a root for $T$. Sort the children
of each node of $T$, by increasing size of their
corresponding subtree. Embed $T$ into $G$
with the embedding algorithm, choosing
at all times the rightmost point as the root
of the next subtree. Suppose that $e$
is used in this embedding. Let $e:=(p,q)$ so that
$u$ is embedded into $p$ and $v$ is embedded
into $q$ (note that $u$ and $v$ are vertices of $T$).

Suppose that the subtree rooted at $v$ has at least two nodes.
In the algorithm, we embedded this subtree into a set of at least two points. We chose
a convex hull point ($q$), of this set visible from $p$ 
to embed $v$.  In this case we may choose another
convex hull point visible from $p$ to embed $v$ and continue
with the algorithm. Note that $(p,q)$ is no longer
used in the final embedding.

Suppose that $v$ is a leaf, and that $v$ 
has a sibling $v'$  whose subtree has 
at least two nodes. Then we may change the order
of the children of $u$ so that $e$ is no longer
used in the embedding, or if it is, then $v'$ is 
embedded into $q$, but then we proceed as above.

Suppose that $v$ is a leaf, and that all its siblings
are leaves. The subtree rooted at $u$ is 
a star. We choose a point distinct
from $p$ and $q$ in the point set where this
subtree is embedded, and embed $u$ into this
point. Afterwards we join it to the remaining points.
This produces an embedding that avoids $e$.

Assume then, that $v$ is a leaf and that it
has no siblings. We distinguish the following cases:

\begin{enumerate}
    \item {\bf $u$ has no siblings.}
      In this case, the subtree rooted at the parent of $u$ is a path
      of length two. It is always possible to
      embed this subtree without using $e$. See Figure~\ref{fig:path}.

    \item {\bf $u$ has a sibling $u'$ whose subtree is not an edge.}
      We may change the order of the siblings of $u$, with respect
      to  their parent, so that the subtree rooted
      at $u'$ 
      will be embedded into the point set containing $p$ and $q$.
      In the initial order---increasing by size of their
      corresponding subtrees---$u'$
      is after $u$. We may assume that in
      the new ordering, the order
      of the siblings of $u$ before it, stays  the same. 
      Therefore $p$ is the rightmost point
      of the set into which the subtree rooted
      at $u'$ will be embedded. Embed $u'$ into $p$.
      Either we find an embedding not using $e$, or this embedding
      falls into one of the cases considered before.

    \item {\bf $u$ has at least one sibling, all whose corresponding subtrees are edges}
      
      Suppose that $u$ has no grandparent; then $T$ is equal to $T_n$ and
      $n$ is odd. Let $w$ be the parent of $u$. Embed $w$
      into $p$. Let $p_1,\dots, p_{n-1}$ be the points
      of $G$ different from $p$ sorted counter-clockwise
      by angle around $p$; choose $p_1$ so that the
      angle between two consecutive points is less
      than $\pi$. Let $u_1,\dots,u_{(n-1)/2}$ be the
      neighbors of $w$. Embed each $u_i$ into $p_{2i-1}$ and
      its child into $p_{2i}$. If $q$ equals $p_{2j-1}$
      for some $j$ then embed $u_j$ into $p_{2j}$
      and its child into $p_{2j-1}$. This embedding
      avoids $e$.

      Suppose that $w$ is the grandparent of $u$ and let $p'$ be the point
      into which  $w$ is embedded. Let $S$ be  the point set
      into which the subtree rooted at the parent of 
      $u$ is embedded. Note that $S$
      has an odd number of points.
      We replace the embedding as follows. Sort $S$ counter-clockwise
      by angle around $p'$. Call a point \emph{even} if it 
      has an even number of points before it in this ordering.
      Call a point \emph{odd} if it 
      has an odd number of points before it in this ordering.
      If $e$ is incident to an odd point, then we embed
      the parent of $u$ into this point.
      The remaining subtree rooted at $u$ can
      be embedded without using $e$. If the endpoints
      of $e$ are both even, between them there is an odd point.
      We embed the parent of $u$ into this point. 
      The remaining vertices can
      be embedded without using $e$ (see Figure~\ref{fig:evenodd}).
      
\end{enumerate}
\end{proof}

\begin{figure}
\begin{center}
  \includegraphics[width=0.8\textwidth]{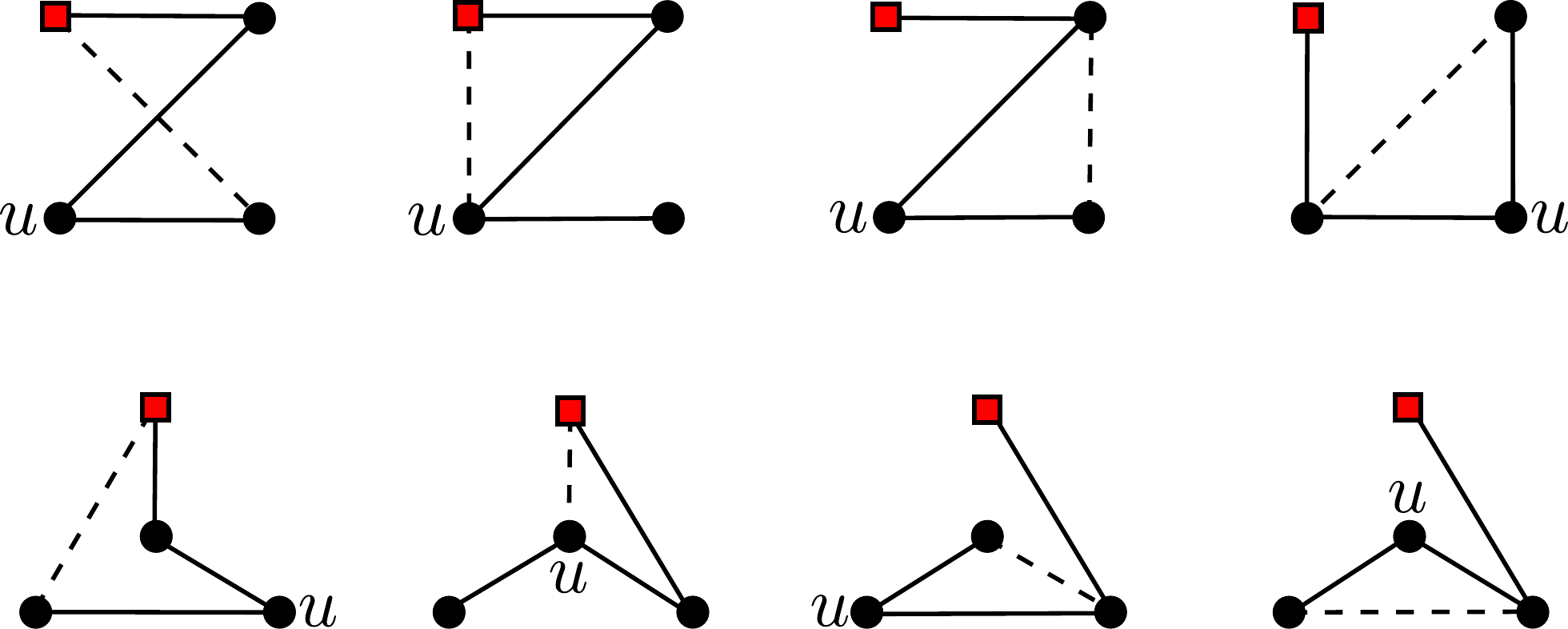}
  \caption{The embedding of a path of length three. The grandparent
    of $u$ is highlighted and the forbidden edge is dashed.}\label{fig:path}
\end{center}
\end{figure}

\begin{figure}
\begin{center}
  \includegraphics[width=0.8\textwidth]{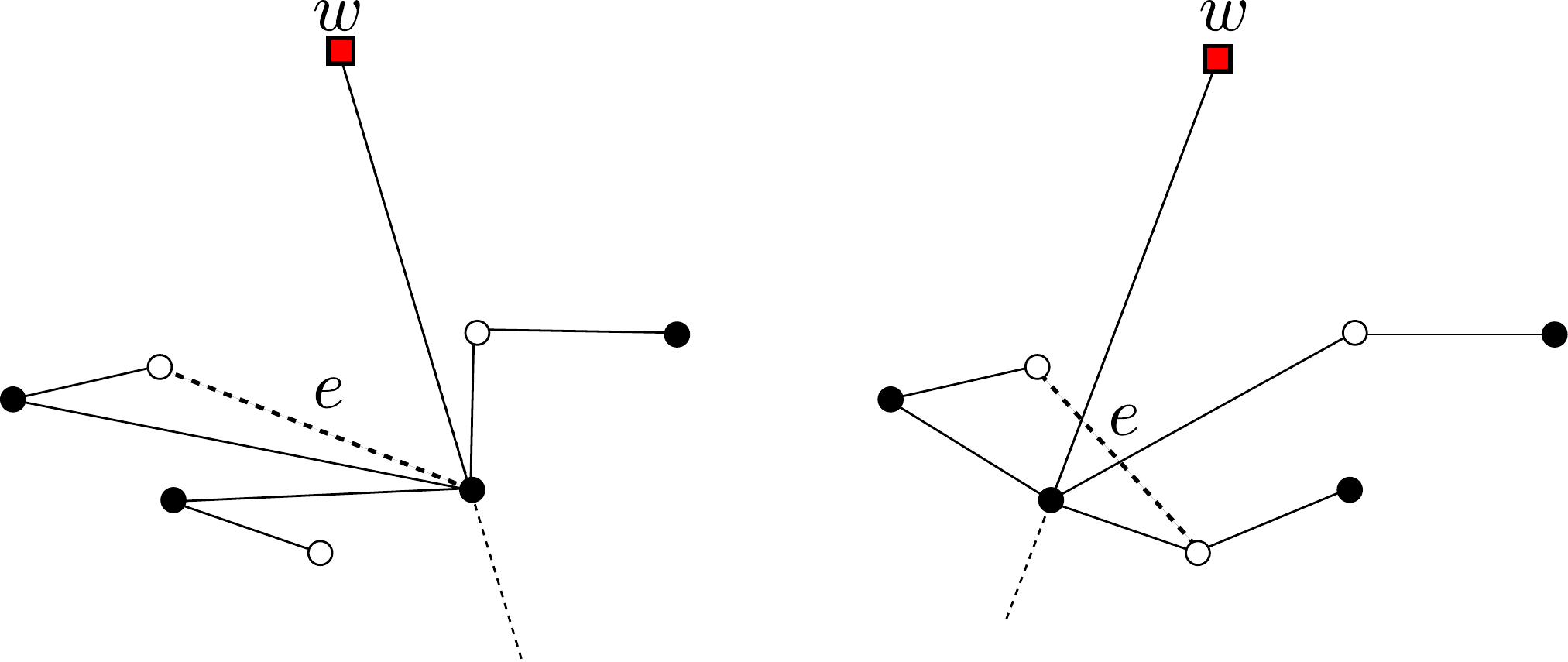}
  \caption{The two sub-cases, when $u$ has a grandparent $w$, and
    all the subtrees of its children are edges. Odd points
  are painted in black and even points in white. The
  forbidden edges are dashed.} \label{fig:evenodd}
\end{center}
\end{figure}

The upper bound  of $f(n,n)\le 3$ of Theorem~\ref{thm:k=n}
follows directly from Lemma~\ref{lem:conf3}.
Now we prove in Lemma~\ref{lem:few_0edges} and
Theorem~\ref{thm:2convex}, that if $G$ is a convex
geometric graph, at least three edges
are needed to forbid some tree on $n$
vertices.

\begin{lemma} \label{lem:few_0edges}
Let $T$ be a tree on $n$ vertices. If $G$ is
a convex geometric graph, then $T$ can be 
embedded into $G$ using less
than $\frac{n}{2}$ convex hull edges of $G$.
\end{lemma}
\begin{proof} 
If $T$ is a star,
then any embedding of $T$ into $G$
uses only two convex hull edges.
If $T$ is a path then it can be embedded
into $G$ using at most two convex hull edges.
Therefore, we may assume that $T$ is neither a star nor a path.

Since $T$ is not a path, it has a
vertex of degree at least three.
Choose this vertex as the root.
Since $T$ is not 
a star, the root has a child whose subtree
has at least two nodes. Sort the 
children of $T$ so that this node is
first. Embed $T$ into $G$ with
the embedding algorithm. 

Let $u$ and
$v$ be vertices of $T$, so that
$u$ is the parent of $v$. Suppose
that the subtree rooted at $v$ has at least
two nodes. Then in the embedding
algorithm we have at least two choices to embed $v$ 
once the ordering of the children of $u$ has been chosen.
At least one of which is such that $(u,v)$ 
is not embedded into a convex hull edge.
Therefore, we may assume that the embedding is such that all the 
convex hull edges used are incident to a leaf.

Since the first child of the root is not a leaf,
there is at most one convex hull edge incident
to the root in the embedding. Note that any 
vertex of $T$, other than the root,
is incident to at most one convex hull edge 
in the embedding. If $n/2$ or 
more convex hull edges are used, then there are at least 
$n/2$ non-leaf vertices, each adjacent to a leaf. These
vertices must be all the vertices in $T$ and there
are only $n/2$ such pairs ($n$ must also be even). 
Therefore every non-leaf vertex has at most one child 
which is a leaf. In particular 
the root has at most one child which is a leaf.
Since the root was chosen of degree at least three it has a child which
is not a leaf nor the first child; we place this
vertex last in the ordering of the children of the root.
The leaf adjacent to the root can no longer be a
convex hull edge and the embedding uses less
than $n/2$ convex hull edges.
\end{proof}

\begin{theorem}\label{thm:2convex}
If $G$ is a convex geometric graph and has at most two 
forbidden edges, then any tree on $n$ vertices can
be embedded into $G$, without using a forbidden
edge.
\end{theorem}
\begin{proof}
Let $f_0$ be an embedding
given by Lemma \ref{lem:few_0edges},
of $T$ into $G$. For $0\le i \le n$, let $f_i$ be the embedding
produced by rotating $f_0$, $i$ places
to the right. Assume that in each
of these rotations at least one forbidden
edge is used, as otherwise we are done.
Let $e_1,\dots, e_m$ be the edges
of $T$ that are mapped to a forbidden
edge in some rotation. Assume that the two forbidden edges are an
$l$-edge and an $r$-edge respectively.

Suppose that $l\neq r$.  
Then, each edge of $T$ can be 
embedded into a forbidden edge at most once in
all of the $n$ rotations. Thus $m\ge n$. This
is a contradiction, since $T$ has $n-1$ edges.

Suppose that $l=r$.
Then, each of the $e_i$ is mapped
twice to a forbidden edge. Thus $m \ge n/2$.
By Lemma \ref{lem:few_0edges},
$f_0$ uses less than $n/2$ convex hull edges. Therefore, $l$ 
and $r$ must be greater than $0$. But
a set of $n/2$ or more $r$-edges, with $r>0$, must contain a pair of
edges that cross.
And we are done, since $f_0$ is an embedding.
\end{proof}

\section{Bounds on $f(n,k)$}

In this section we prove Theorem~1. First we show the upper bound.

\begin{lemma}\label{lem:conf3}
If $G$ is a convex geometric graph, 
then forbidding three consecutive convex hull edges of $G$
forbids the embedding of $T_n$.
\end{lemma}
\begin{proof}
Recall that $T_n$ comes from subdividing a star, let
$v$ be the non leaf vertex of this star.
Let $(p_1,p_2),(p_2,p_3),(p_3,p_4)$ be the forbidden
edges, in clockwise order around the convex hull
of $G$.
Note that in any embedding of $T_n$ into $G$, an
edge incident to a leaf of $T_n$, must be embedded
into a convex hull edge. Thus, the leaves of $T_n$ nor its neighbors
can be embedded into $p_2$ or $p_3$, without using a forbidden edge.
Thus, $v$ must be embedded into $p_2$ or $p_3$. Without
loss of generality assume that $v$ is embedded into $p_2$.
But then, the embedding must use $(p_2,p_3)$ or $(p_3,p_4)$
\end{proof}

\begin{lemma}\label{lem:conf2_2}
If $G$ is a convex geometric graph,
then forbidding any three pairs of consecutive convex hull edges of $G$
forbids the embedding of $T_n$.
\end{lemma}
\begin{proof}
Let $p_1, p_2$ and $p_3$ be the vertices in the middle
of the three pairs of consecutive forbidden edges of $G$.
Note that a leaf of $T_n$, nor its neighbor can be
embedded into $p_1$, $p_2$ or $p_3$, without
using a forbidden edge. But at most two points
do not fall into this category.
\end{proof}

\begin{lemma}
$ f(n,k) \le 2 \frac{n(n-2)}{k-2}$
\end{lemma}
\begin{proof}
Let $G$ be a complete convex geometric graph. 
We forbid every $r$-edge of $G$ for 
$r=0,\dots,\left \lceil 2\frac{n-2}{k-2}-2 \right \rceil$.
Note that, in total we are forbidding at most
$n \left ( \left \lceil 2\frac{n-2}{k-2}-2 \right \rceil+1 \right )\le 2\frac{n(n-2)}{k-2}$ edges.
As every subset of points of $G$ is 
in convex position, it suffices to show that every
induced subgraph $H$ of $G$ on $k$ 
vertices is in one of the two configurations 
of Lemma \ref{lem:conf3} and \ref{lem:conf2_2}.

Assume then, that $H$ does not contain three consecutive
forbidden edges in its convex hull nor three pairs of consecutive
forbidden edges in its convex hull. $H$ has at most
two (non-adjacent) pairs of consecutive forbidden
edges in its convex hull. Therefore every forbidden 
edge of $H$ in its convex hull---with the exception of at most two---must be preceded
by an $\ell$-edge (of $G$), with $\ell>\left \lceil 2\frac{n-2}{k-2}-2 \right \rceil$. $H$
contains at least $\frac{k-2}{2}$ of these edges. 
The points separated by these edges amount to more than 
 $\frac{k-2}{2}\left \lceil 2\frac{n-2}{k-2}-2 \right \rceil\ge n-k$ points of $G$. This
is a contradiction, since together with the $k$ points
of $H$ this is strictly more than $n$.

\end{proof}

Now, we show the lower bound of Theorem \ref{thm:main}.


\begin{lemma}
$f(n,k)\ge \left( \frac{1}{2} \right )\frac{n^2}{k-1}-\frac{n}{2}$
\end{lemma}
\begin{proof}
Let $F$ be a set of edges whose removal from $G$ forbids
some $k$-tree. Let $H:=G\setminus F$. Note that $H$
contains no complete $K_k$ as a subgraph, otherwise
any $k$-tree can be embedded in this subgraph \cite{optimal}. 
By Tur{\'a}n's Theorem \cite{turan}, $H$ cannot contain 
more than $\left (\frac{k-2}{k-1} \right )\frac{n^2}{2}$ edges.
Thus $F$ must have size at least 
$\left( \frac{1}{2} \right )\frac{n^2}{k-1}-\frac{n}{2}$.
\end{proof}

\section*{Acknowledgments}
{\small Part of this work was done at the ``First Workshop in Combinatorial
Optimization at Cinvestav''. It was continued during 
a visit of L.F. Barba, R. Fabila-Monroy, J. Lea\~nos and
G. Salazar to Abacus research center\footnote{ABACUS, CONACyT grant EDOMEX-2011-C01-165873}. }



\small 
\bibliographystyle{abbrv}

\bibliography{erdos-sos}



\end{document}